\documentclass[a4paper,oneside,12pt]{article}

\usepackage{amsmath,amsfonts,amscd,amssymb}
\usepackage{longtable,geometry}
\usepackage[english]{babel}
\usepackage[utf8]{inputenc}
\usepackage[active]{srcltx}
\usepackage[T1]{fontenc}
\usepackage{graphicx}
\usepackage{pstricks}
\usepackage{bbm}
\usepackage{mathtools}
\usepackage{enumerate}
\usepackage{MnSymbol}
\usepackage{stmaryrd}
\usepackage{nicefrac}
\usepackage{calrsfs}
\usepackage{enumitem}
 \usepackage{rotating}
\usepackage{xcolor}
\usepackage{framed}

\colorlet{shadecolor}{blue!15}

\usepackage{amsthm}
\theoremstyle{plain}
\newtheorem{theorem}{Theorem}[section]

\newtheorem{lemma}[theorem]{Lemma}

\theoremstyle{remark}

\newcommand{\be}[1]{\begin{equation}\label{#1}}
\newcommand{\ee}{\end{equation}}
\numberwithin{equation}{section}

\newcommand{\ba}[1]{\begin{align}\label{#1}}
\newcommand{\ea}{\end{align}}
\numberwithin{equation}{section}

\newcommand{\ben}{\begin{equation*}}
\newcommand{\een}{\end{equation*}}
\numberwithin{equation}{section}

\newcommand{\calC}{\mathcal{C}}

\newcommand{\calS}{\mathcal{S}}

\newcommand{\bbE}{\mathbb{E}}

\newcommand{\bbP}{\mathbb{P}}

\newcommand{\bbZ}{\mathbb{Z}}

\newcommand{\rk}[1]{\bgroup\color{red}%
  \par\medskip\hrule\smallskip%
  \noindent\textbf{#1}%
  \par\smallskip\hrule\medskip\egroup}

\usepackage{dsfont}
\usepackage{hyperref}
\usepackage{mathtools}
\mathtoolsset{showonlyrefs}

\newcommand{\lr}[1][]{\stackrel{#1}\longleftrightarrow}

\newcommand{\nlr}[1][]{\overset{#1}{\not\longleftrightarrow}}

\title{A new proof of the sharpness of the phase transition for
  Bernoulli percolation on $\mathbb Z^d$}
\author{Hugo Duminil-Copin and Vincent Tassion}
\date{\today}

\begin{document}
\maketitle

\begin{abstract}
We provide a new proof of the sharpness of the phase transition for
nearest-neighbour Bernoulli percolation. More precisely, we show that
\begin{itemize}
\item for $p<p_c$, the probability that the origin is connected by an
  open path to distance $n$ decays exponentially fast in $n$.
\item for $p>p_c$, the probability that the origin belongs to an
  infinite cluster satisfies the mean-field lower bound $\theta(p)\ge
  \tfrac{p-p_c}{p(1-p_c)}$.
\end{itemize}
This note presents the argument of \cite{DumTas15}, which is
valid for long-range Bernoulli percolation (and for the Ising model) on arbitrary transitive graphs in the simpler framework
of nearest-neighbour Bernoulli percolation on $\bbZ^d$. \end{abstract}

\section{Statement of the result}
\paragraph{Notation.}
Fix an integer $d\ge 2$. We consider the
$d$-dimensional hypercubic lattice $(\bbZ^d,\bbE^d)$. Let
$\Lambda_n=\{-n,\ldots,n\}^d$, and let $\partial \Lambda_n:=\Lambda_n\setminus\Lambda_{n-1}$ be its
vertex-boundary. Throughout this note, \emph{$S$
always stands for a finite set of vertices containing the origin}.
Given such a set, we denote its edge-boundary by $\Delta S$, defined
by all the edges $\{x,y\}$ with $x\in S$ and $y\notin S$.
\medbreak
Consider the
Bernoulli bond percolation measure $\bbP_p$ on $\{0,1\}^{\bbE^d}$ for which each edge of
$\bbE^d$ is declared {\em open} with probability $p$ and closed otherwise,
independently for different edges. 
\medbreak
Two vertices $x$ and
$y$ are {\em connected in} $S\subset V$ if there exists a path of
vertices $(v_k)_{0\le k\le K}$ in $S$ such that $v_0=x$, $v_K=y$, and
$\{v_k,v_{k+1}\}$ is open for every $0\le k<K$. We denote this event
by $x\lr[S] y$. If $S=\bbZ^d$, we drop it from the notation. We set
$0\lr \infty$ (resp.\@
$0\lr \partial\Lambda_n$) if $0$ is connected to
infinity (resp.\@ $0$ is connected to a vertex in
$\partial\Lambda_n$).

\paragraph{Phase transition.}
A new idea of this paper is to use a different definition of the critical
parameter than the standard one. This new definition relies
on the following quantity. For $p\in[0,1]$ and $0\in
S\subset \bbZ^d$, define
\begin{equation}
  \label{eq:16}
  \varphi_p(S):=p\sum_{\{x,y\}\in \Delta S} \bbP_p[0\lr[S] x]
\end{equation}
and introduce the following quantities:
\begin{align}
  \label{eq:17}
  \tilde p_c&:= \sup\big\{p\in[0,1] \text{ s.t.\@ there exists a finite
    set $0\subset S \subset \bbZ^d$ with $\varphi_p(S)<1$}\big\},\\
    p_c&:=\sup\{p\text{ s.t. $\bbP_p[0\longleftrightarrow\infty]=0$}\}\nonumber.
\end{align}
We are now in a position to state our main result.
\begin{theorem}\label{thm:perco} For any $d\ge2$, $\tilde p_c=p_c$. Furthermore,
\begin{enumerate}
\item\label{item:1}For $p<p_c$, there exists $c=c(p)>0$ such that for every $n\ge 1$,
$$\bbP_p[0\longleftrightarrow \partial\Lambda_n]\le e^{-c n}.$$
\item\label{item:2} For $p>p_c$, \[\bbP_p[0\longleftrightarrow\infty]\ge
  \frac{p-p_c}{p(1-p_c)}.\]
\end{enumerate}
\end{theorem}

\noindent\textbf{Remarks.}
\begin{enumerate}
\item We refer to \cite{DumTas15} for a detailed bibliography, and for a version of the proof valid in greater generality. The aim of this note is to provide a proof in the simplest framework.
\item Theorem~\ref{thm:perco} was proved by Aizenman and
  Barsky~\cite{AizBar87} in the more general framework of long-range
  percolation. In their proof, they consider an additional parameter
  $h$ corresponding to an external field, and
  they derive the results from differential inequalities satisfied by
  the thermodynamical quantities of the model. A different proof, based on the
  geometric study of the pivotal edges, was obtained at the same time by
  Menshikov~\cite{Men86}. These two proofs are also presented in \cite{Gri99a}. 
  \item In the definition of $\tilde p_c$, the set of parameters $p$
  such that there exists a finite set $0\subset S \subset \bbZ^d$ with
  $\varphi_p(S)<1$ is an open subset of $[0,1]$. Thus, $\tilde p_c$ do
  not belong to this set. 
   \begin{center}
  \includegraphics[width=9cm]{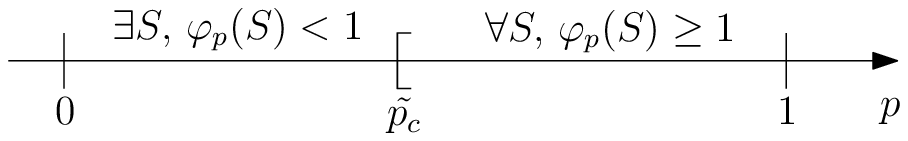}
\end{center}
We obtain that the expected size of the
  cluster of the origin satisfies that for every $p>
   p_c$, $$\sum_{x\in\bbZ^d}\bbP_p[0\longleftrightarrow x]\ge\sum_{n\ge 0} \varphi_p(\Lambda_n) =+\infty.$$
   \item Since $\varphi_p(\{0\})=2dp$, we obtain $p_c\ge 1/{2d}$.
\item Item~\ref{item:2} provides a mean-field lower bound for the infinite cluster density.
\item Theorem~\ref{thm:perco} implies that $p_c\le 1/2$ on $\bbZ^2$.
  Combined with Zhang's argument  \cite[Lemma 11.12]{Gri99a}, this
  shows that $p_c=1/2$.
\end{enumerate}

\section{Proof of the theorem}
It is sufficient to show Items 1 and 2 with $p_c$ replaced by $\tilde p_c$ (since it immediately implies the equality $p_c=\tilde p_c$).
\subsection{Proof of Item~\ref{item:1}}
The proof of Item~\ref{item:1} can be derived from
the BK-inequality~\cite{BerKes85}. We present here an exploration
argument, similar to the one in~\cite{Ham57}, which does not rely on the
BK-inequality. Let $p<\tilde p_c$. By definition, one can fix a
finite set $S$ containing the origin, such that $ \varphi_p(S)<1$. Let
$L>0$ such that $S\subset \Lambda_{L-1}$.

Let $k\ge 1$ and assume
that the event $0\lr\partial\Lambda_{kL}$ holds. Let $$\calC=\{z\in S:0\lr[S]z\}.$$
Since $S\cap\partial
\Lambda_{kL}=\emptyset$, there exists an edge $\{x,y\}\in\Delta S$ such
that  the following events occur:
\begin{itemize}[noitemsep,nolistsep]\item $0$ is connected to $x$ in $S$,
\item $\{x,y\}$ is open,
\item $y$ is connected to $\partial \Lambda_{kL}$ in $\calC^c$.
\end{itemize}
Using first the union bound, and then a decomposition with respect to possible values of $\calC$, we find
\begin{align}
  &\bbP_p[0\lr\partial\Lambda_{kL}]\\
  &\le \sum_{\{x,y\}\in
    \Delta S}\sum_{C\subset S}\bbP_p\big[\{0\lr[S]x,\calC=C\}\cap\{\{x,y\}\text{ is open}\}\cap\{y\lr[\bbZ^d\setminus C]\partial
  \Lambda_{kL}\}\big]\\
&= p\sum_{\{x,y\}\in
    \Delta S}\sum_{C\subset S}\bbP_p\big[0\lr[S]x,\calC=C\big]\bbP_p\big[y\lr[\bbZ^d\setminus C]\partial
  \Lambda_{kL}\big].\end{align}
In the second line, we used the fact that the three events depend on different sets of edges and are therefore independent. Since $y\in
\Lambda_L$, one can bound $\bbP_p[y\lr[\bbZ^d\setminus C]\partial
\Lambda_{kL}]$ by  $\bbP_p[0\lr\partial  \Lambda_{(k-1)L}]$ in the last expression. Hence, we find
\begin{equation}
  \bbP_p[0\lr\partial\Lambda_{kL}]\le\varphi_p(S)\bbP_p\big[y\lr\partial
  \Lambda_{(k-1)L}\big]
\end{equation}
which by induction gives
\begin{equation}
  \label{eq:18}
  \bbP_p[0\lr\partial\Lambda_{kL}]\le\varphi_p(S)^{k-1}.
\end{equation}
This proves the desired exponential decay.

\subsection{Proof of Item~\ref{item:2}}
\label{sec:proof-item}
Let us start by the following lemma providing a differential
inequality valid for every $p$.
 \begin{lemma}\label{lem:meanField}
Let $p\in[0,1]$ and $n\ge 1$,
  \begin{equation}
    \label{eq:4}
    \frac d{dp} \bbP_p[0\lr\partial\Lambda_n]\ge \frac1{p(1-p)}\cdot
    \inf_{\substack{S\subset\Lambda_n\\0\in S}}\varphi_p(S) \cdot\big(1-\bbP_p[0\lr\partial\Lambda_n]\big).
  \end{equation}
  \end{lemma}
  Let us first see how it implies Item~\ref{item:2} of
  Theorem~\ref{thm:perco}. Integrating the differential inequality~\eqref{eq:4} between
  $\tilde p_c$ and $p>\tilde p_c$ implies that for every $n\ge1$,
  $\bbP_p[0\lr\partial\Lambda_n]\ge \frac{p-\tilde p_c}{p(1-\tilde p_c)}$. By letting $n$
  tend to infinity, we obtain the desired lower bound on $\bbP_p[0\longleftrightarrow\infty]$.

\begin{proof}[Proof of Lemma~\ref{lem:meanField}]
  Recall that $\{x,y\}$ is pivotal for the configuration
$\omega$ and the event $\{0\longleftrightarrow\partial\Lambda_n\}$ if
$\omega_{\{x,y\}}\notin \{0\longleftrightarrow\partial\Lambda_n\}$ and
$\omega^{\{x,y\}}\in \{0\longleftrightarrow\partial\Lambda_n\}$. (The
configuration $\omega_{\{x,y\}}$, resp.\@ $\omega^{\{x,y\}}$,  coincides
with $\omega$ except that the edge $\{x,y\}$ is closed, resp.\@
open.) By Russo's formula (see \cite[Section 2.4]{Gri99a}), we have
  \begin{align}
        \frac d{dp}
        \bbP_p[0\lr\partial\Lambda_n]&=\sum_{e\subset \Lambda_n}
        \bbP_p\big[\text{$e$ is pivotal}\big]\\
        &=\frac1{1-p}\sum_{e\subset \Lambda_n}
        \bbP_p\big[\text{$e$ is pivotal},\,0\nlr\partial\Lambda_n\big].
  \end{align}
  Define the following random subset of $\Lambda_n$:
\[\calS:=\{x\in\Lambda_n\text{ such that } x\not\longleftrightarrow
\partial\Lambda_n\}.\]
The boundary of $\calS$ corresponds to the outmost blocking surface (which can be obtained by exploring from the outside the set of vertices connected to the boundary). When $0$ is not connected to $\partial\Lambda_n$, the set $\calS$ is
always a subset of $\Lambda_n$ containing the origin. By summing over
the possible values for $\calS$, we obtain
 \begin{equation}
        \frac d{dp}
        \bbP_p[0\lr\partial\Lambda_n]=\frac1{1-p}\sum_{\substack{S\subset
            \Lambda_n\\0\in S}} \sum_{e\subset \Lambda_n}
        \bbP_p\big[\text{$e$ is pivotal},\,\calS=S\big]
  \end{equation}
  Observe that on the event $\calS=S$, the pivotal edges are 
  the edges $\{x,y\}\in \Delta S$ such that $0$ is connected to $x$
  in $S$. This implies that
  \begin{equation}
    \label{eq:19}
    \frac d{dp}
        \bbP_p[0\lr\partial\Lambda_n]=\frac1{1-p}\sum_{\substack{S\subset
            \Lambda_n\\0\in S}} \sum_{\{x,y\}\in \Delta S}
        \bbP_p\big[0\lr[S]x,\, \calS=S\big].
  \end{equation}
  The event $\{\calS=S\}$ is measurable with respect to the
  configuration outside $S$ and is therefore independent of $\{0\lr[S]x\}$. We obtain
   \begin{align}
    \frac d{dp}
        \bbP_p[0\lr\partial\Lambda_n]&=\frac1{1-p}\sum_{\substack{S\subset
            \Lambda_n\\0\in S}} \sum_{\{x,y\}\in \Delta S}
        \bbP_p\big[0\lr[S]x\big] \bbP_p\big[\calS=S\big]\\
        &=\frac1{p(1-p)}\sum_{\substack{S\subset
            \Lambda_n\\0\in S}}\varphi_p(S)\bbP_p\big[\calS=S\big]\\
        &\ge \frac1{p(1-p)}\inf_{\substack{S\subset
            \Lambda_n\\0\in S}}\varphi_p(S)\cdot\bbP_p\big[0\nlr\partial\Lambda_n],
  \end{align}
  as desired.
\end{proof}

\paragraph{Acknowledgments} This work was supported by a grant from the Swiss FNS and the NCCR SwissMap also founded by the swiss NSF. \bibliographystyle{alpha}
\bibliography{biblicomplete}
\small\begin{flushright}
\textsc{D\'epartement de Math\'ematiques}
  \textsc{Universit\'e de Gen\`eve}
  \textsc{Gen\`eve, Switzerland}
  \textsc{E-mail:} \texttt{hugo.duminil@unige.ch}, \texttt{vincent.tassion@unige.ch}
\end{flushright}
\end{document}